\documentclass[11pt, leqno]{amsart}
\usepackage{amssymb}

\newtheorem{thm}{Theorem}[section]
\newtheorem{lem}[thm]{Lemma}
\newtheorem{cor}[thm]{Corollary}
\newtheorem{prop}[thm]{Proposition}

\theoremstyle{definition}
\newtheorem{defn}[thm]{Definition}

\theoremstyle{remark}

\newtheorem{exam}[thm]{Example}

\numberwithin{equation}{section}

\newcommand{\ord}{\text{ord}}
\newcommand{\z}{{\mathbb Z}}
\newcommand{\q}{{\mathbb Q}}

\newcommand{\gen}{\text{gen}}

\newcommand{\p}{{\mathfrak p}}

\newcommand{\ring}{\mathfrak o}

\newcommand{\ba}{\mathbf a}
\newcommand{\be}{\mathbf e}
\newcommand{\bx}{\mathbf x}
\newcommand{\bu}{\mathbf u}
\newcommand{\bv}{\mathbf v}
\newcommand{\bw}{\mathbf w}

\newcommand{\bz}{\mathbf z}

\newcommand{\ad}{{\mathbb A}}

\begin{document}

\title[Integral quadratic polynomials]{Representations of integral quadratic polynomials}

\author{Wai Kiu Chan}
\address{Department of Mathematics and Computer Science, Wesleyan University, Middletown CT, 06459, USA}
\email{wkchan@wesleyan.edu}

\author{Byeong-Kweon Oh}
\address{Department of Mathematical Sciences and Research Institute of Mathematics, 
Seoul National University,
 Seoul 151-747, Korea}
\email{bkoh@snu.ac.kr}
\thanks{This work of the second author was supported by the National Research Foundation of Korea(NRF) grant funded by the Korea government(MEST) (No. 20110027952)} 



\begin{abstract}
In this paper, we study the representations of integral quadratic polynomials.  Particularly, it is shown that there are only finitely many equivalence classes of positive ternary universal integral quadratic polynomials, and that there are only finitely many regular ternary triangular forms.  A more general discussion of integral quadratic polynomials over a Dedekind domain inside a global field is also given.
\end{abstract}

\maketitle

\section{Introduction}

For a polynomial $f(x_1, \ldots, x_n)$ with rational coefficients and an integer $a$, we say that $f$ represents $a$ if the diophantine equation
\begin{equation} \label{1steqn}
f(x_1, \ldots, x_n) = a
\end{equation}
is soluble in the integers.  The {\em representation problem} asks for a complete determination of the set of integers represented by a given polynomial.   This problem is considered to be untractable in general in view of Matiyasevich's negative answer to Hilbert's tenth problem \cite{ma}.  Moreover, Jones \cite{j} has shown that whether a general single diophantine equation of degree four or higher is soluble in the positive integers is already undecidable.  However, the linear and the quadratic cases have been studied extensively.  The linear case is elementary and its solution is a consequence of the Euclidean algorithm.  For the quadratic case, the representation problem for homogeneous quadratic polynomials, or quadratic forms in other words, has a long history and it still garners a lot of attention from mathematicians across many areas.  For accounts of more recent development of the subject, the readers are referred to the surveys \cite{h, sp} and the references therein.  In this paper, we will discuss a couple of questions which are related to the representation problem of quadratic polynomials in general, namely {\em universality} and {\em regularity}, which we will explain below.

A quadratic polynomial $f(\bx) = f(x_1, \ldots, x_n)$ can be written as
$$f(\bx) = Q(\bx) + L(\bx) + c$$
where $Q(\bx)$ is a quadratic form, $L(\bx)$ is a linear form, and $c$ is a constant.  Unless stated otherwise {\em we assume that $Q$ is positive definite}.   This in particular implies that there exists a unique vector $\bv \in \q^n$ such that $L(\bx) = 2B(\bv, \bx)$, where $B$ is the bilinear form such that $B(\bx, \bx) = Q(\bx)$.  As a result,
$$f(\bx) = Q(\bx + \bv) - Q(\bv) + c \geq -Q(\bv) + c,$$
and so $f(\bx)$ attains an absolute minimum on $\z^n$.  We denote this minimum by $m_f$ and will simply call it the minimum of $f(\bx)$.  We call $f(\bx)$ {\em positive} if $m_f \geq 0$.

In this paper, we call a quadratic polynomial $f(\bx)$ {\em integral} if it is integer-valued, that is, $f(\bx) \in \z$ for all $\bx \in \z^n$.   A positive integral quadratic polynomial $f(\bx)$ is called {\em universal} if it represents all nonnegative integers.  Positive definite universal integral quadratic forms have been studied for many years by many authors and have become a popular topic in the recent years.  It is known that positive definite universal integral quadratic forms must have at least four variables, and there are only finitely many equivalence classes of such universal quadratic forms in four variables.  Moreover, a positive definite integral quadratic form is universal if and only if it represents all positive integers up to 290 \cite{bh}.   However, Bosma and Kane \cite{bk} show that this kind of finiteness theorem does not exist for positive integral quadratic polynomials in general.  More precisely, given any finite subset $T$ of $\mathbb N$ and a positive integer $n \not \in T$, Bosma and Kane construct explicitly a positive integral quadratic polynomial with minimum 0 which represents every integer in $T$ but not $n$.

An integral quadratic polynomial is called {\em almost universal} if it represents all but finitely many positive integers.  A classical theorem of Tartakovski \cite{t} implies that a positive definite integral quadratic form in five or more variables is almost universal provided it is universal over $\z_p$ for every prime $p$.  An effective procedure for deciding whether a positive definite integral quadratic form in four variables is almost universal is given in \cite{bo}.

Unlike positive definite universal or almost universal quadratic forms, positive universal and almost universal integral quadratic polynomials do exist in three variables.  One well-known example of universal quadratic polynomial is the sum of three triangular numbers
$$\frac{x_1(x_1+1)}{2} + \frac{x_2(x_2+1)}{2} + \frac{x_3(x_3 + 1)}{2}.$$
Given positive integers $a_1, \ldots, a_n$, we follow the terminology used in \cite{co} and call the polynomial
$$\Delta(a_1, \ldots, a_n): = a_1\frac{x_1(x_1 + 1)}{2} + \cdots + a_n\frac{x_n(x_n + 1)}{2}$$
a triangular form.   There are only seven universal ternary triangular forms and they were found by Liouville in 1863 \cite{li}.  Bosma and Kane \cite{bk} have a simple criterion--the Triangular Theorem of Eight--to determine the universality of a triangular form: a triangular form is universal if and only if it represents the integers 1, 2, 4, 5, and 8.  In \cite{co}, the present authors give a complete characterization of triples of positive integers $a_1, a_2, a_3$ for which $\Delta(a_1, a_2, a_3)$ are almost universal.   Particularly, it is shown there that there are infinitely many almost universal ternary triangular forms.  Almost universal integral quadratic polynomials in three variables that are mixed sums of squares and triangular numbers are determined in \cite{ch} and \cite{ks}.

Two quadratic polynomials $f(\bx)$ and $g(\bx)$ are said to be {\em equivalent} if there exists $T \in \text{GL}_n(\z)$ and $\bx_0 \in \z^n$ such that
\begin{eqnarray} \label{equiv}
g(\bx) = f(\bx T + \bx_0).
\end{eqnarray}
One can check readily that this defines an equivalence relation on the set of quadratic polynomials, and equivalent quadratic polynomials represent the same set of integers. In Section \ref{universal}, we will prove the following finiteness result on almost universal integral quadratic polynomials in three variables.  It, in particular, implies that given a nonnegative integer $k$, there are only finitely many almost universal ternary triangular forms that represent all integers $\geq k$.

\begin{thm} \label{thmin3}
Let $k$ be a nonnegative integer.  There are only finitely many equivalence classes of positive integral quadratic polynomials in three variables that represent all integers $\geq k$.
\end{thm}

An integral polynomial is called {\em regular} if it represents all the integers that are represented by the polynomial itself over $\z_p$ for every prime $p$ including $p = \infty$ (here $\z_\infty = \mathbb R$ by convention).  In other words, $f(\bx)$ is regular if
\begin{equation} \label{hasse}
(\ref{1steqn}) \mbox{ is soluble in $\z_p$ for every $p \leq \infty$ } \Longrightarrow (\ref{1steqn}) \mbox{ is soluble in $\z$}.
\end{equation}
Watson \cite{w1, w2} showed that up to equivalence there are only finitely many primitive positive definite regular integral quadratic forms in three variables.   A list containing all possible candidates of equivalence classes of these regular quadratic forms is compiled by Jagy, Kaplansky, and Schiemann in \cite{jks}. This list contains 913 candidates and all but twenty two of them are verified to be regular.  Recently Oh \cite{o} verifies the regularity of eight of the remaining twenty two forms.  As a first step to understand regular quadratic polynomials in three variables, we prove the following in Section \ref{regularpoly}.

\begin{thm} \label{regular3}
There are only finitely many primitive regular triangular forms in three variables.
\end{thm}

A quadratic polynomial $f(\bx)$ is called {\em complete} if it takes the form
$$f(\bx) = Q(\bx) + 2B(\bv, \bx) + Q(\bv) = Q(\bx + \bv).$$
Every quadratic polynomial is complete after adjusting the constant term suitably.   In Section \ref{coset}, we will describe a geometric approach of studying the arithmetic of complete quadratic polynomials.  In a nut shell, a complete integral quadratic polynomial $f(\bx)$ is just a coset $M + \bv$ of an integral $\z$-lattice $M$ on a quadratic $\q$-space with a quadratic map $Q$, and solving the diophantine equation $f(\bx) = a$ is the same as finding a vector $\be$ in $M$ such that $Q(\be + \bv) = a$.  The definition of the class number of a coset will be introduced, and it will be shown in Section \ref{coset} that this class number is always finite and can be viewed as a measure of obstruction of the local-to-global implication in (\ref{hasse}).

In the subsequent sections, especially in Section \ref{coset}, we will complement our discussion with the geometric language of quadratic spaces and lattices.  Let $R$ be a PID.  If $M$ is a $R$-lattice on some quadratic space over the field of fractions of $R$ and $A$ is a symmetric matrix, we shall write ``$M \cong A$" if $A$ is the Gram matrix for $M$ with respect to some basis of $M$.  The discriminant of $M$ is the determinant of one of its Gram matrices.  An $n \times n$ diagonal matrix with $a_1, \ldots, a_n$ as its diagonal entries is written as $\langle a_1, \ldots, a_n \rangle$.  Any other unexplained notation and terminology in the language of quadratic spaces and lattices used in this paper can be found in \cite{ca}, \cite{ki}, and \cite{om}.

\section{Universal Ternary Quadratic Polynomials} \label{universal}

We start this section with a technical lemma which will be used in the proof of Theorem \ref{thmin3}.

\begin{lem} \label{2variables}
Let $q(\bx)$ be a positive definite binary quadratic form and $b$ be the associated bilinear form.  For $i = 1, \ldots, t$, let $f_i(\bx) = q(\bx) + 2b(\bw_i, \bx) + c_i$ be a positive integral quadratic polynomial with quadratic part $q(\bx)$.  For any integer $k \geq 0$, there exists a positive integer $N\geq k$, bounded above by a constant depending only on $q(\bx)$, $k$, and $t$, such that $N$ is not represented by $f_i(\bx)$ for every $i = 1, \ldots, t$.
\end{lem}
\begin{proof}
Let $d$ be the discriminant of $q(\bx)$.  Choose odd primes $p_1 < \cdots < p_t$ such that $-d$ is a nonresidue mod $p_i$ for all $i$.  Then for every $i = 1, \ldots, t$, $q(\bx)$ is anisotropic $\z_{p_i}$-unimodular.  In particular, $q(\bx) \in \z_{p_i}$, and hence $2b(\bw_i, \bx)$ as well,  are in $\z_{p_i}$ for all $\bx \in \z_{p_i}^2$.  This implies that $\bw_i \in \z_{p_i}^2$ and so $q(\bw_i) \in \z_{p_i}$.    Let $N$ be the smallest positive integer satisfying $N \geq k$ and
$$N \equiv p_i + c_i - q(\bw_i) \mod p_i^2, \quad \mbox{ for } i = 1, \ldots, t.$$
Then for every $i$, $\ord_{p_i}(N - c_i + q(\bw_i)) = 1$ and so $N - c_i + q(\bw_i)$ is not represented by $q(\bx + \bw_i)$ over $\z_{p_i}$.  Thus $N$ is not represented by $f_i(\bx)$.
\end{proof}

A positive ternary quadratic polynomial $f(\bx) = Q(\bx) + 2B(\bv, \bx) + m$ is called {\em Minkowski reduced}, or simply {\em reduced},  if its quadratic part is Minkowski reduced and it attains its minimum at the zero vector.  This means that the quadratic part $Q(\bx)$ is of the form $\bx A \bx^t$, where $A$ is a Minkowski reduced symmetric matrix.  So, if $\be_1, \be_2, \be_3$ is the standard basis for $\z^3$, then $Q(\be_1) \leq Q(\be_2) \leq Q(\be_3)$.  Also, $Q(\bx) + 2B(\bv, \bx) \geq 0$ for all $\bx \in \z^3$, and hence
\begin{equation} \label{inequality}
2\vert B(\bv, \be_i) \vert \leq Q(\be_i) \mbox{ for } i = 1, 2, 3.
\end{equation}

\begin{lem} \label{reduced}
Every positive ternary quadratic polynomial is equivalent to a reduced ternary quadratic polynomial.
\end{lem}
\begin{proof}
Let $f(\bx)$ be a positive ternary quadratic polynomial.  It follows from reduction theory that there exists $T \in \text{GL}_n(\z)$ such that the quadratic part of $f(\bx T)$ is Minkowski reduced.  If $f(\bx T)$ attains its minimum at $\bx_0$, then the polynomial $g(\bx):= f(\bx T + \bx_0)$, which is equivalent to $f(\bx)$, is reduced.
\end{proof}

\begin{lem} \label{reduction}
Let $Q(\bx)$ be a positive definite reduced ternary quadratic form.  Then for any $(x_1, x_2, x_3) \in \z^3$,
$$Q(x_1\be_1 + x_2\be_2 + x_3\be_3) \geq \frac{1}{6}(Q(\be_1)x_1^2 + Q(\be_2)x_2^2 + Q(\be_3)x_3^2).$$
\end{lem}
\begin{proof} Let $C_{ij} = Q(\be_i)Q(\be_j) - B(\be_i, \be_j)^2$, which is positive if $i \neq j$ because $Q(\bx)$ is reduced.  For any permutation $i, j, k$ of the integers $1, 2, 3$, we have
$$Q(\be_k)C_{ij} \leq Q(\be_1)Q(\be_2)Q(\be_3) \leq 2 D,$$
where $D$ is the discriminant of $Q$. Now, by completing the squares,
\begin{eqnarray*}
Q(x_1\be_1 + x_2\be_2 + x_3\be_3) & \geq & Q(\be_i)(x_i + \cdots)^2 + \frac{C_{ij}}{Q(\be_j)}(x_j + \cdots )^2 + \frac{D}{C_{ij}} x_k^2\\
    & \geq & \frac{Q(\be_k)}{2} x_k^2.
\end{eqnarray*}
Thus
$$3(Q(x_1\be_1 + x_2\be_2 + x_3\be_3)) \geq \frac{1}{2}(Q(\be_1)x_1^2 + Q(\be_2)x_2^2 + Q(\be_3)x_3^2),$$
and the lemma follows immediately.
\end{proof}

We are now ready to prove Theorem \ref{thmin3}.

\begin{proof}[Proof of Theorem \ref{thmin3}]
Let $k$ be a fixed nonnegative integer.  By virtue of Lemma \ref{reduced}, it suffices to show that there are only finitely many reduced positive ternary integral quadratic polynomials which represent all positive integers $\geq k$.  By adjusting the constant terms of these quadratic polynomials, we may assume that their minimum is 0.

Let $f(\bx) = Q(\bx) + 2B(\bv, \bx)$ be a reduced positive ternary integral quadratic polynomial with minimum $0$.  Let $\be_1, \be_2, \be_3$ be the standard basis for $\z^3$. For simplicity, for each $i = 1, 2, 3$, we denote $Q(\be_i)$ by $\mu_i$ and $B(\bv, \be_i)$ by $w_i$.  Furthermore, for $i \neq j$, let $a_{ij}$ be $B(\be_i, \be_j)$.   We assume throughout below that $f(\bx)$ represents all integers $\geq k$.  The proof will be complete if we can show that $\mu_3$ is bounded above by a constant depending only on $k$. From now on, $(x_1, x_2, x_3)$ always denotes a vector in $\z^3$.

By (\ref{inequality}) and Lemma \ref{reduction},
\begin{eqnarray*}
f(x_1, x_2, x_3) & \geq & \sum_{i=1}^3 \left(\frac{1}{6} \mu_i x_i^2 - 2\vert w_i x_i \vert\right)\\
    & \geq & \sum_{i = 1}^3 \mu_i \left(\frac{1}{6} x_i^2 - \vert x_i \vert \right),
\end{eqnarray*}
and so if $\vert x_3 \vert \geq 9$, we have
$$f(x_1, x_2, x_3) \geq -\frac{3}{2}\mu_1 - \frac{3}{2}\mu_2 + \frac{9}{2} \mu_3 \geq \frac{3}{2}\mu_3.$$

Suppose that $\vert x_3 \vert \leq 8$.  Since $2\vert a_{12} \vert \leq \mu_1$, one obtains $\frac{\mu_1}{2}x_1^2 + 2a_{12}x_1x_2 + \frac{\mu_2}{2}x_2^2 \geq 0$ for all $(x_1, x_2) \in \z^2$.  So, if $\vert x_2 \vert \geq 22$, then
\begin{eqnarray*}
f(x_1, x_2, x_3) & \geq & \frac{\mu_1}{2}x_1^2 + 2(a_{13}x_3 + w_1)x_1 + \frac{\mu_2}{2} x_2^2 + 2(a_{23}x_3 + w_2)x_2 + f(0,0,x_3) \\
    & \geq & - \frac{81}{2}\mu_1 + 44\mu_2\\
    & \geq & \frac{7}{2}\mu_2.
\end{eqnarray*}
Let us assume further that $\vert x_2 \vert \leq 21$.  If, in addition, $\vert x_1 \vert \geq 31$, then
\begin{eqnarray*}
f(x_1, x_2, x_3) & = & \mu_1 x_1^2 + 2(a_{12}x_2 + a_{13}x_3 + w_1)x_1 + f(0, x_2, x_3)\\
    & \geq & \mu_1(x_1^2 - 30 \vert x_1\vert)\\
    & \geq & 31 \mu_1.
\end{eqnarray*}
Therefore, we have
$$f(x_1, x_2, x_3) \geq \gamma(f): = \min \left \{ \frac{3}{2}\mu_3, \frac{7}{2}\mu_2, 31\mu_1 \right \}$$
unless
$$\vert x_1 \vert \leq 30, \quad \vert x_2 \vert \leq 21, \,\, \mbox{ and }\,\, \vert x_3 \vert \leq 8.$$
In particular, this means that there are at most $61\times 43 \times 17$ choices of $(x_1, x_2, x_3)$ for which $f(x_1, x_2, x_3) < \gamma(f)$, and thus there are at most $61\times 43\times 17$ distinct positive integers less than $\gamma(f)$ which may be represented by $f$.  So, if $\gamma(f) \geq 61\times 43 \times 17 + 2 + k$, then $f(x_1, x_2, x_3)$ does not represent at least one integer among $k + 1, k + 2, \ldots, k + 61\times 43 \times 17 + 1$.  Consequently,
$$\frac{3}{2}\mu_1 \leq \gamma(f) \leq k + 61 \times 43 \times 17 + 1.$$
Let $\eta$ be the smallest positive integer satisfying
$$43 \times 17 \times [2(15 + \sqrt{225 + k + \eta}) + 1] < \eta.$$
Suppose that $\frac{3}{2}\mu_2 > k + \eta$.  Let $s$ be a positive integer $\leq k + \eta$.  If $f(x_1, x_2, x_3) = s$, then $\vert x_2 \vert \leq 21$ and $\vert x_3 \vert \leq 8$; thus, as shown before,
\begin{eqnarray*}
f(x_1, x_2, x_3) & =  & \mu_1 x_1^2 + 2(a_{12}x_2 + a_{13}x_3 + w_1)x_1 + f(0, x_2, x_3)\\
    & \geq & \mu_1(x_1^2 - 30 \vert x_1\vert)\\
    & \geq & x_1^2 - 30 \vert x_1 \vert.
\end{eqnarray*}
So, if $\vert x_1 \vert > 15 + \sqrt{225 + k + \eta}$, then $f(x_1, x_2, x_3) > k +  \eta$.  Therefore, the number of vectors $(x_1, x_2, x_3) \in \z^3$ satisfying $k + 1 \leq f(x_1, x_2, x_3) \leq k + \eta$ is not bigger than
$$43 \times 17 \times [2(15 + \sqrt{225 + k +  \eta}) + 1],$$
which is strictly less than $\eta$.  This is impossible, which means that
$$\mu_2 \leq \frac{2(k + \eta)}{3}.$$

Recall that if $\vert x_3 \vert \geq 9$, then $f(x_1, x_2, x_3) \geq \frac{3}{2}\mu_3$.  It follows from Lemma \ref{2variables} that there exists a positive integer $N \geq k$ which is not represented by $f(x_1, x_2, t)$ for any integer $t \in [-8, 8]$, and this $N$ is bounded above by a constant depending only on $k, \mu_1, \mu_2$, and $a_{12}$ (note that $2\vert a_{12}\vert \leq \mu_1$).  This means that whenever $f(x_1, x_2, x_3) = N$, we must have $\vert x_3 \vert \geq 9$ and so
$$\mu_3 \leq \frac{2N}{3}.$$
This completes the proof.
\end{proof}

\section{Regular Ternary Triangular Forms} \label{regularpoly}

A triangular form $\Delta(\alpha_1, \ldots, \alpha_n)$ is said to be {\em primitive} if $\gcd(\alpha_1, \ldots, \alpha_n) = 1$.  Its discriminant, denoted $d(\Delta)$, is defined to be the product $\alpha_1\cdots \alpha_n$.  By completing the squares, it is easy to see that $\Delta(\alpha_1, \ldots, \alpha_n)$ represents an integer $m$ if and only if the equation
\begin{equation} \label{3to2}
\alpha_1(2x_1 + 1)^2 + \cdots + \alpha_n(2x_n + 1)^2 = 8m + (\alpha_1 + \cdots + \alpha_n)
\end{equation}
is soluble in $\z$.  Let $M$ be the $\z$-lattice with quadratic map $Q$ and an orthogonal basis $\{\be_1, \ldots, \be_n\}$ such that $M \cong \langle 4\alpha_1, \ldots, 4\alpha_n \rangle$. Then (\ref{3to2}) is soluble in $\z$ if and only if $8m + (\alpha_1 + \cdots + \alpha_n)$ is represented by the coset $M + \bv$, where $\bv = (\be_1 + \cdots + \be_n)/2$, that is, there exists a vector $\bx \in M$ such that $Q(\bx + \bv) = 8m + (\alpha_1 + \cdots + \alpha_n)$.

Let $p$ be an odd prime.  If $M_p$ is the $\z_p$-lattice $\z_p\otimes M$, then $M_p + \bv = M_p$.  Therefore,  (\ref{3to2}) is soluble in $\z_p$ if and only if $M_p$ represents $8m + (\alpha_1 + \cdots + \alpha_n)$.  In particular, $\Delta(\alpha_1, \ldots, \alpha_n)$ is universal over $\z_p$ if and only if $M_p$ is universal.

\begin{lem}\label{at2}
A primitive triangular form is universal over $\z_2$.
\end{lem}
\begin{proof}
It suffices to prove that for an odd integer $\alpha$, the polynomial $\alpha x(x + 1)/2$ is universal over $\z_2$.  But this is clear by the Local Square Theorem \cite[63:1]{om} or \cite[Lemma 1.6, page 40]{ca}.
\end{proof}

\begin{lem} \label{atodd}
Let $p$ be an odd prime and $\alpha, \beta, \gamma$ be $p$-adic units.  Then over $\z_p$,
\begin{enumerate}
\item[(1)] $\Delta(\alpha,\beta)$ represents all integers $m$ for which $8m + \alpha + \beta \not \equiv 0$ mod $p$;
\item[(2)] $\Delta(\alpha, \beta)$ is universal if $\alpha + \beta  \equiv 0 \mod p$;
\item[(3)] $\Delta(\alpha, \beta, \gamma)$ is universal.
\end{enumerate}
\end{lem}
\begin{proof}
The binary $\z_p$-lattice $\langle \alpha, \beta \rangle$ represents all $p$-adic units \cite[92:1b]{om}.  Therefore, it represents all integers $m$ for which $8m + \alpha + \beta \not \equiv 0 \mod p$.  This proves (1).

In (2), the condition on $\alpha$ and $\beta$ implies that the $\z_p$-lattice $\langle \alpha,\beta \rangle$ is isometric to the hyperbolic plane which is universal.  For (3), it follows from \cite[92:1b]{om} that any unimodular $\z_p$-lattice of rank at least three is universal.
\end{proof}

Recall that a triangular form is regular if it represents all positive integers that are represented by the triangular form itself over $\z_p$ for all primes $p$.  For example, every universal triangular form is regular.   The following lemma is a ``descending trick" which transforms a regular ternary triangular form to another one with smaller discriminant.

\begin{lem} \label{watson}
Let $q$ be an odd prime and $a, b, c$ be positive integers which are not divisible by $q$.  Suppose that $\Delta(a, q^rb, q^sc)$ is regular, with $1 \leq r \leq s$.  Then $\Delta(q^{2-\delta}a, q^{r - \delta}b, q^{s - \delta}c)$ is also regular, where $\delta = \min\{2, r\}$.
\end{lem}
\begin{proof}
It suffices to show that $\Delta(q^2a, q^rb, q^sc)$ is regular.  Suppose that $m$ is a positive integer represented by $\Delta(q^2a, q^rb, q^sc)$ over $\z_p$ for all primes $p$.  Then the equation
\begin{equation} \label{1}
8m + (q^2a + q^rb + q^sc) = q^2a (2x_1 + 1)^2 + q^rb (2x_2 + 1)^2 + q^sc (2x_3 + 1)^2
\end{equation}
is soluble in $\z_p$ for every prime $p$.  Since $q$ is odd, we can say that
\begin{equation}\label{2}
8m + (q^2a + q^rb + q^sc) = a(2x_1 + 1)^2 + q^rb (2x_2 + 1)^2 + q^sc (2x_3 + 1)^2
\end{equation}
is also soluble in $\z_p$ for every prime $p$.  Notice that $q^2 \equiv 1$ mod 8, and so $8m + (q^2a + q^rb + q^sc) = 8m' + (a + q^rb + q^sc)$ for some integer $m'$.  Thus, the regularity of $\Delta(a, q^rb, q^sc)$ implies that (\ref{2}) is soluble in $\z$.  Let $(x_1, x_2, x_3) \in \z^3$ be a solution to (\ref{2}).  Then $(2x_1 + 1)$ must be divisible by $q$ because $q \mid m$ by (\ref{1}), and we can write $(2x_1 + 1)$ as $q(2y_1 + 1)$ for some $y_1 \in \z$.  So $(y_1, x_2, x_3)$ is an integral solution to (\ref{1}), which means that $m$ is in fact represented by $\Delta(q^2a, q^rb, q^sc)$.
\end{proof}

The following lemma will be used many times in the subsequent discussion. It is a reformulation of \cite[Lemma 3]{kko}.

\begin{lem}\label{kkolemma}
Let $T$ be a finite set of primes and $a$ be an integer not divisible by any prime in $T$.  For any integer $d$, the number of integers in the set $\{d, a + d, \ldots, (n-1)a + d \}$ that are not divisible by any prime in $T$ is at least
$$n\frac{\tilde{p}-1}{\tilde{p} + t - 1} - 2^t + 1,$$
where $t = \vert T \vert$ and $\tilde{p}$ is the smallest prime in $T$.
\end{lem}

For the sake of convenience, we say that a ternary triangular form $\Delta(\alpha, \beta, \gamma)$ {\em behaves well} if the unimodular Jordan component of the $\z_p$-lattice $\langle \alpha, \beta, \gamma \rangle$ has rank at least two, or equivalently, $p$ does not divide at least two of $\alpha, \beta$, and $\gamma$.  For a ternary triangular form $\Delta$, we can rearrange the variables so that $\Delta = \Delta(\mu_1, \mu_2, \mu_3)$ with $\mu_1 \leq \mu_2 \leq \mu_3$.  Collectively, we call these $\mu_i$ the successive minima of $\Delta$.

In what follows, an inequality of the form $A \ll B$ always means that there exists a constant $k > 0$ such that $\vert A \vert \leq k\vert B\vert$.  A real-valued function in several variables is said to be bounded if its absolute value is bounded above by a constant independent of the variables.

\begin{prop} \label{well}
There exists an absolute constant $C$ such that if $\Delta$ is a primitive regular ternary triangular form which behaves well at all odd primes, then $d(\Delta) \leq C$.
\end{prop}
\begin{proof}
Let $\mu_1\leq \mu_2\leq \mu_3$ be the successive minima of $\Delta$, and let $M$ be the $\z$-lattice $\langle 4\mu_1, 4\mu_2, 4\mu_3 \rangle$.  Let $T$ be the set of odd primes $p$ for which $M_p$ is not split by the hyperbolic plane.  Then $T$ is a finite set.  Let $t$ be the size of $T$, $\tilde{p}$ be the smallest prime in $T$, and $\omega = (\tilde{p} + t - 1)/(\tilde{p} - 1)$.   Note that, since $\tilde{p} \geq 2$, we have $\omega \leq t + 1$.  Let $\eta = (\mu_1 + \mu_2 + \mu_3)$ and $\mathfrak T$ be the product of primes in $T$.  It follows from Lemmas \ref{at2} and \ref{atodd} and the regularity of $\Delta$ that $\Delta$ represents every positive integer $m$ for which  $8m + \eta$ is relatively prime to $\mathfrak T$.

By Lemma \ref{kkolemma}, there exists a positive integer $k_1 < (t+1)2^t$ such that $8k_1 + \eta$ is relatively prime to $\mathfrak T$.  Therefore, $k_1$ is represented by $\Delta$ and hence
$$\mu_1 \leq (t+1)2^t \ll t2^t.$$
For any positive integer $n$, the number of integers between 1 and $n$ that are represented by the triangular form $\Delta(\mu_1)$ is at most $2\sqrt{n}$.  Therefore, by virtue of Lemma \ref{kkolemma}, if $n \geq 4(t+1)^2 + 3(t+1)2^t$, there must be a positive integer $k_2 \leq n$ such that $8k_2 + \eta$ is relatively prime to $\mathfrak T$ and $k_2$ is not represented by $\Delta(\mu_1)$.  This implies that
$$\mu_2 \leq 4(t+1)^2 + 3(t+1)2^t \ll t2^t.$$

Let $\mathfrak A$ be the product of primes in $T$ that do not divide $\mu_1\mu_2$.  Following the argument in \cite[page 862]{e}, we find that there must be an odd prime $q$ outside $T$ such that $-\mu_1\mu_2$ is a nonresidue mod $q$ and $q \ll (\mu_1\mu_2)^{\frac{7}{8}} {\mathfrak A}^{\frac{1}{4}}$.  Since $\mathfrak A \leq \mathfrak T$, we have
$$q \ll (\mu_1\mu_2)^{\frac{7}{8}} {\mathfrak T}^{\frac{1}{4}} \ll (t2^t)^{\frac{7}{8}}{\mathfrak T}^{\frac{1}{4}}.$$
Fix a positive integer $m \leq q^2$ such that
$$8m + \mu_1 + \mu_2 \equiv q \mod q^2.$$
For any integer $\lambda$, $8(m + \lambda q^2) + \mu_1 + \mu_2$ is not represented by the binary lattice $\langle \mu_1, \mu_2 \rangle$, which means that $m + \lambda q^2$ is not represented by $\Delta(\mu_1, \mu_2)$.   However, by Lemma \ref{kkolemma}, there must be a positive integer $k_3 \leq (t+1)2^t$ such that $8q^2k_3 + 8m + \eta$ is relatively prime to $\mathfrak T$.  Then $m + q^2k_3$ is an integer represented by $\Delta$ but not by $\Delta(\mu_1,\mu_2)$.  As a result,
$$\mu_3 \leq m + q^2k_3 \ll (t 2^t)^{\frac{11}{4}} {\mathfrak T}^{\frac{1}{2}},$$
and hence
$$\mathfrak T \leq d(\Delta) = \mu_1\mu_2\mu_3 \ll (t 2^t)^{\frac{19}{4}} {\mathfrak T}^{\frac{1}{2}}.$$
Since $\mathfrak T$, a product of $t$ distinct primes, grows at least as fast as $t!$, the above inequality shows that $t$, and hence $\mathfrak T$ as well, must be bounded.  This means that $d(\Delta)$ is also bounded.
\end{proof}

Starting with a primitive regular ternary triangular form $\Delta$, we may apply Lemma \ref{watson} successively at suitable odd primes and eventually obtain a primitive regular ternary triangular form $\overline{\Delta}$ which behaves well at all odd primes.  It is also clear from Lemma \ref{watson} that $d(\overline{\Delta})$ divides $d(\Delta)$.  Let $\ell$ be an odd prime divisor of $d(\Delta)$.  If $\ell$ divides $d(\overline{\Delta})$, then $\ell$ is bounded by Proposition \ref{well}.  So we assume from now on that $\ell$ does not divide $d(\overline{\Delta})$.  Our next goal is to bound $\ell$.

When we obtain $\overline{\Delta}$ from $\Delta$, we may first apply Lemma \ref{watson} at all primes $p$ not equal to $\ell$.  So, there is no harm to assume from the outset that $\Delta$ behaves well at all primes $p \neq \ell$.   Then, by Lemma \ref{watson}, $\Delta$ can be transformed to a primitive regular ternary triangular form $\tilde{\Delta} = \tilde{\Delta}(a, \ell^2b, \ell^2c)$, with $\ell \nmid abc$, which behaves well at all primes $p \neq \ell$.  Since further application of Lemma \ref{watson} to $\tilde{\Delta}$ results in the triangular form $\overline{\Delta}$, therefore all the prime divisors of $d(\tilde{\Delta})$, except $\ell$, are bounded.

Let $\tilde{T}$ be the set of odd prime divisors of $d(\tilde{\Delta})$ that are not $\ell$.   By Lemmas \ref{at2} and \ref{atodd}, we see that
$\tilde{\Delta}$ represents all positive integers $m$ for which $8m + a + \ell^2b + \ell^2c$ is relatively prime to every prime in $\tilde{T}$ and $(8m + a + \ell^2b + \ell^2c)a$ is a quadratic residue modulo $\ell$.

In order to find integers represented by $\tilde{\Delta}$, we need a result which is a slight generalization of Proposition 3.2 and Corollary 3.3 in \cite{e}.  Let $\chi_1, \ldots, \chi_r$ be Dirichlet characters modulo $k_1, \ldots, k_r$, respectively, $u_1, \ldots, u_r$ be values taken from the set $\{\pm 1\}$, and $\Gamma$ be the least common multiple of $k_1, \ldots, k_r$.  Given a nonnegative number $s$ and a positive number $H$, let $S_s(H)$ be the set of integers $n$ in the interval $(s, s + H)$ which satisfy the conditions
$$\chi_i(n) = u_i \quad \mbox{ for } i = 1, \ldots, r
\text{ and } \gcd(n, X) = 1,$$
where $X$ is a positive integer relatively prime to $\Gamma$.

\begin{prop} \label{eresult}
Suppose that $\chi_1, \ldots, \chi_r$ are independent.  Let $h = \min \{H : S_s(H) > 0 \}$ and $\omega(\Gamma)$ denote the number of distinct prime divisor of $\Gamma$.  Then
\begin{equation} \label{e3.2}
S_s(H) = 2^{-r}\frac{\phi(\Gamma X)}{\Gamma X}H + O\left(H^{\frac{1}{2}}\Gamma^{\frac{3}{16} + \epsilon} X^\epsilon\right),
\end{equation}
and if $r \leq \omega(\Gamma) + 1$, we have
\begin{equation} \label{e3.3}
h \ll \Gamma^{\frac{3}{8} + \epsilon}X^\epsilon,
\end{equation}
where $\phi$ is the Euler's phi-function and the implied constants in both \textnormal{(\ref{e3.2})} and \textnormal{(\ref{e3.3})} depend only on $\epsilon$.
\end{prop}
\begin{proof}
We may proceed as in the proofs for Proposition 3.2 and Corollary 3.3 in \cite{e}, but notice that \cite[Lemma 3.1]{e} remains valid if we replace ``$0 < n < H$" by ``$s < n < s + H$" in the summations since Burgess's estimate for character sums  \cite[Theorem 2]{b} holds for any interval of length $H$.
\end{proof}

\begin{lem} \label{primel}
The prime $\ell$ is bounded.
\end{lem}
\begin{proof}
Let $\tilde{\mu_1} \leq \tilde{\mu_2}$ be the first two successive minima of $\tilde{\Delta}$.    Let $s = a + \ell^2 b + \ell^2 c$ and write $s = 2^\kappa s_0$ with $2 \nmid s_0$.  Suppose that $\kappa \geq 3$.  We apply Proposition \ref{eresult} to the quadratic residue mod $\ell$ character $\chi_\ell$, taking $\epsilon = 1/8$ and $X$ to be the product of the primes in $\tilde{T}$.  So,  there is a positive integer $h \ll \ell^{\frac{1}{2}}$  such that $\chi_\ell(h + 2^{\kappa - 3}s_0) = \chi_\ell(2a)$ and $h + 2^{\kappa - 3} s_0$ is not divisible by any prime in $\tilde{T}$.  Then $\tilde{\Delta}$ represents $h$ and hence $\tilde{\mu_1} \ll \ell^{\frac{1}{2}}$.

If $\kappa < 3$, then we apply Proposition \ref{eresult} again but this time to $\chi_\ell$ and possibly the mod 4 character $\left(\frac{-1}{*} \right)$ and the mod 8 character $\left(\frac{2}{*} \right)$.   We obtain a positive integer $n > s_0$ such that $\chi_\ell(n) = \chi_\ell(2^{\kappa} a)$, $n$ is not divisible by any prime in $\tilde{T}$, $n \equiv s_0$ mod $2^{3 - \kappa}$, and $n - s_0 \ll \ell^{\frac{1}{2}}$.  Then we can write $2^\kappa n = 8m + s$, where $m$ is represented by $\tilde{\Delta}$ and $m \ll \ell^{\frac{1}{2}}$.  So, $\tilde{\mu_1} \ll \ell^{\frac{1}{2}}$ in this case as well.

Now, for any $H > 0$, the number of integers in the interval $(s, s + H)$ that are represented by the triangular form $\Delta(\tilde{\mu_1})$ is equal to $O(\sqrt{H})$.  Thus, by Proposition \ref{eresult} and an argument similar to the one above, we must have $\tilde{\mu_2} \ll \ell^{\frac{1}{2}}$.  Then $\ell^2 \leq \tilde{\mu_1}\tilde{\mu_2} \ll \ell$, and hence $\ell$ is bounded.
\end{proof}

We now present the proof of Theorem \ref{regular3} which asserts that there are only finitely many primitive regular ternary triangular forms.

\begin{proof}[Proof of Theorem \ref{regular3}]
Let $\Delta$ be a primitive regular ternary triangular form, and $\mu_1\leq \mu_2\leq \mu_3$ be its successive minima.  It suffices to show that these successive minima are bounded.  Let $S$ be the set of odd prime divisors of $d(\Delta)$.  It follows from Proposition \ref{well} and Lemma \ref{primel} that all the primes in $S$ are bounded.  Let $\mathfrak S$ be the product of these primes.  It is clear from Lemma \ref{at2} and Lemma \ref{atodd}(3) that $\Delta$ represents $\mathfrak S$ over $\mathbb Z_p$ for all $p \not \in S$.  Also, Lemma \ref{atodd}(1) (if $\mu_1 + \mu_2 \not \equiv 0$ mod $p$) or Lemma \ref{atodd}(2) (if $\mu_1 + \mu_2 \equiv 0$ mod $p$) shows that $\Delta$ represents $\mathfrak S$ over $\mathbb Z_p$ for all primes $p \in S$.  Consequently, $\Delta$ represents $\mathfrak S$ over $\mathbb Z_p$ for all primes $p$.  Since $\Delta$ is regular, it must represent $\mathfrak S$.  This shows that $\mu_1$ is bounded.

Let $q_1$ be the smallest odd prime not dividing $3\mu_1\mathfrak S$, and $q_2$ be the smallest odd prime not dividing $q_1\mu_1\mathfrak S$ for which $8q_2 \mathfrak S\mu_1 + \mu_1^2$ is a nonresidue mod $q_1$.  Such $q_2$ exists because there are at least two nonresidues mod $q_1$.  Note that $q_2\mathfrak S$ is represented by $\Delta$ but not by $\Delta(\mu_1)$.  Therefore, $\mu_2$ is also bounded.

Now, let $q_3$ be the smallest odd prime not dividing $\mathfrak S$ for which $-\mu_1\mu_2$ is a nonresidue mod $q_3$, and $q_4$ be the smallest odd prime not dividing $\mathfrak S$ which satisfies
$$-8q_4 \mathfrak S \equiv \mu_1 + \mu_2 + q_3 \mod q_3^2.$$
Then $q_4\mathfrak S$ is represented by $\Delta$ but not by $\Delta(\mu_1, \mu_2)$, which means that $\mu_3$ is bounded.  This completes the proof.
\end{proof}

\section{Representations of Cosets} \label{coset}

In the previous sections we have seen some connection between the diophantine aspect of quadratic polynomials and the geometric theory of quadratic  spaces and lattices.  In this section we will amplify this connection  by describing a geometric approach of a special, but yet general enough for most practical purpose, family of quadratic polynomials.   Since it will not present any additional difficulty, we shall consider quadratic polynomials over global fields and the Dedekind domains inside.  For simplicity, the quadratic map and its associated bilinear form on any quadratic space will be denoted by $Q$ and $B$ respectively.

Now, unless stated otherwise, $F$ is assumed to be a global field of characteristic not 2 and $\ring$ is a Dedekind domain inside $F$ defined by a Dedekind set of places $\Omega$ on $F$ (see, for example, \cite[\S 21]{om}).   We call a quadratic polynomial $f(\bx)$ over $F$ in variables $\bx = (x_1, \ldots, x_n)$ {\em complete} if $f(\bx) = (\bx + \bv)A (\bx + \bv)^t$, where $A$ is an $n\times n$ nonsingular symmetric matrix over $F$ and $\bv \in F^n$.  It is called {\em integral} if $f(\bx) \in \ring$ for all $\bx \in \ring^n$.  Two quadratic polynomials $f(\bx)$ and $g(\bx)$ are said to be {\em equivalent} if there exist $T \in \text{GL}_n(\ring)$ and $\bx_0 \in \ring^n$ such that $g(\bx) = f(\bx T + \bx_0)$.

On the geometric side, an $\ring$-{\em coset} is a set $M + \bv$, where $M$ is an $\ring$-lattice on an $n$-dimensional nondegenerate quadratic space $V$ over $F$ and $\bv$ is a vector in $V$.   An $\ring$-coset $M + \bv$ is called {\em integral} if $Q(M + \bv) \subseteq \ring$, and is {\em free} if $M$ is a free $\ring$-lattice.  Two $\ring$-cosets $M + \bv$ and $N + \bw$ on two quadratic spaces $V$ and $W$, respectively, are said to be {\em isometric}, written $M + \bv \cong N + \bw$, if there exists an isometry $\sigma: V \longrightarrow W$ such that $\sigma(M + \bv) = N + \bw$.   This is the same as requiring $\sigma(M) = N$ and $\sigma(\bv) \in \bw + N$.  For each $\p \in \Omega$, $\ring_\p$-cosets and isometries between $\ring_\p$-cosets are defined analogously.

As in the case of quadratic forms and lattices, there is a one-to-one correspondence between the set of equivalence classes of complete quadratic polynomials in $n$ variables over $F$ and the set of isometry classes of free cosets on $n$-dimensional quadratic spaces over $F$.  Under this correspondence, integral complete quadratic polynomials corresponds to integral free cosets.


\begin{defn}
Let $M + \bv$ be an $\ring$-coset on a quadratic space $V$.  The genus of $M + \bv$ is the set
$$\text{gen}(M + \bv) = \{K + \bw \mbox{ on } V : K_\p + \bw \cong M_\p + \bv \mbox{ for all  } \p \in \Omega\}.$$
\end{defn}

\begin{lem}
Let $M + \bv$ be an $\ring$-coset on a quadratic space $V$ and let $S$ be a finite subset of $\Omega$.  Suppose that an $\ring_\p$-coset $M(\p) + \bx_\p$ on $V_\p$ is given for each $\p \in
S$.  Then there exists an $\ring$-coset $K + \bz$ on $V$ such that
$$K_\p + \bz = \left \{
\begin{array}{ll}
M(\p) + \bx_\p & \mbox{ if $\p \in S$};\\
M_\p + \bv & \mbox{ if $\p \in \Omega \setminus S$}.
\end{array}
\right .$$ \label{crt}
\end{lem}
\begin{proof}
Let $T$ be the set of places $\p \in \Omega \setminus S$ for which $\bv \not \in M_\p$.  Then $T$ is a finite set.  For each $\p \in T$, let $M(\p) = M_\p$ and $\bx_\p = \bv$.  Let $K$ be an $\ring$-lattice on $V$ such that
$$K_\p = \left \{
\begin{array}{ll}
M(\p) & \mbox{ if $\p \in S\cup T$};\\
M_\p & \mbox{ if $\p \in \Omega\setminus (S\cup T)$}.
\end{array}
\right .$$
By the strong approximation property of $V$, there exists $\bz \in V$ such that $\bz \equiv \bx_\p$ mod $M(\p)$ for all $\p
\in S\cup T$, and $\bz \in M_\p$ for all $\p \in \Omega \setminus (S\cup T)$.  Then $K + \bz$ is the desired $\ring$-coset.
\end{proof}

Let $O_\ad(V)$ be the adelization of the orthogonal group of $V$.  Let $\Sigma$ be an element in $O(V)_\ad$.  The $\p$-component of $\Sigma$ is denoted by $\Sigma_\p$.  Given an $\ring$-coset $M + \bv$ on $V$, $\Sigma_\p(M_\p + \bv) = \Sigma_\p(M_\p) = M_\p$ for almost all finite places $\p$.  By Lemma \ref{crt}, there exists an $\ring$-coset $K + \bz$ on $V$ such that $K_\p + \bz = \Sigma_\p(M_\p + \bv)$ for all $\p \in \Omega$. Therefore, we can define
$\Sigma(M + \bv)$ to be $K + \bz$, and so $O(V)_\ad$ acts transitively on $\text{gen}(M + \bv)$.  As a result,
$$\text{gen}(M + \bv) = O_\ad(V) \cdot (M + \bv).$$
Let $O_\ad(M + \bv)$ be the stabilizer of $M + \bv$ in $O_\ad(V)$.  Then the (isometry) classes in $\text{gen}(M + \bv)$ can be
identified with
$$O(V)\setminus O_\ad(V) / O_\ad(M + \bv).$$

The group $O_\ad(M + \bv)$ is clearly a subgroup of $O_\ad(M)$.  For each $\p \in \Omega$,  we have
\begin{eqnarray*}
O(M_\p + \bv) & = & \{ \sigma \in O(V_\p) : \sigma(M_\p) = M_\p \mbox{ and } \sigma(\bv) \equiv \bv \mbox{ mod } M_\p \}\\
& \subseteq & O(M_\p)\cap O(M_\p + \ring_\p \bv).
\end{eqnarray*}

\begin{lem}
For any $\p \in \Omega$, the group index $[O(M_\p) : O(M_\p + \bv)]$ is finite.
\end{lem}
\begin{proof}
There is the natural map
$$O(M_\p)\cap O(M_\p + \ring_\p \bv) \longrightarrow \text{Aut}_{\ring_\p}((M_\p + \ring_\p \bv) / M_\p)$$
whose kernel is precisely $O(M_\p + \bv)$.  Since $(M_\p + \ring_\p \bv)/ M_\p$ is a finite group, the index $[O(M_\p)\cap O(M_\p + \ring_\p \bv) : O(M_\p + \bv)]$ is finite.  But it is known \cite[30.5]{kn} that the index $[O(M_\p) : O(M_\p)\cap O(M_\p + \ring_\p \bv)]$ is always finite.  This proves the lemma.
\end{proof}

Since $M_\p = M_\p + \bv$ for almost all $\p \in \Omega$, the index $[O_\ad(M) : O_\ad(M + \bv)]$ is finite.  The set
$O(V)\setminus O_\ad(V) / O_\ad(M)$ is finite (which is the class number of $M$), hence the set $O(V)\setminus O_\ad(V) / O_\ad(M + \bv)$ is also finite.  Let $h(M + \bv)$ be the number of elements in this set, which is also the number of classes of in $\text{gen}(M +\bv)$.  We call it the {\em class number} of $M + \bv$.

\begin{cor} \label{classnumber}
The class number $h(M + \bv)$ is finite, and $h(M + \bv) \geq h(M)$.
\end{cor}

If we replace the orthogonal groups by the special orthogonal groups in the above discussion, then we have the definitions for the proper genus $\text{gen}^+(M + \bv)$, which can be identified with $O^+(V)\setminus O^+_\ad(V)/O^+_\ad(M + \bv)$, and the proper class number $h^+(M + \bv)$ which is also finite.  Unlike the case of lattices, it is not true in general that $\text{gen}(M + \bv)$ coincides with $\text{gen}^+(M + \bv)$.  The following example illustrates this phenomenon.  It also shows that in general $h(M + \bv)$ and $h(M)$ are not equal.

\begin{exam}
Let $W$ be the hyperbolic plane over $\mathbb Q$, and let $H$ be the $\mathbb Z$-lattice on $W$ spanned by two linear independent isotropic vectors $\be$ and $\mathbf f$ such that $B(\be,\mathbf f) = 1$.  Consider the $\z$-coset $H + \bv$, where $\bv = \frac{1}{p}\be$ for some odd prime $p$.  Suppose that
$\sigma_p$ is an improper isometry of $H_p + \bv$.  Then $\sigma_p$ must send $\be$ to $\epsilon \mathbf f$ and $\mathbf f$ to $\epsilon^{-1}\be$ for some unit $\epsilon$ in $\mathbb Z_p$.
Then
$$\bv = \frac{1}{p}\be \equiv \sigma_p(\bv) \equiv \frac{\epsilon}{p} \mathbf f \mbox{ mod } H_p.$$
This implies that $\frac{1}{p}\be - \frac{\epsilon}{p}\mathbf f$ is in $H_p$, which is absurd.  Therefore, $H_p + \bv$ does not have
any improper isometry and hence $\text{gen}(H + \bv)$ is not the same as $\text{gen}^+(H + \bv)$.

Now, suppose in addition that $p > 3$.  Let $q$ be an integer such that $q \not \equiv \pm 1$ mod $p$.  Let $\bu$ be the vector $\frac{q}{p}\mathbf e$.  Then the coset $H + \bu$ is in $\text{gen}^+(H + \bv)$.  To see this, observe that $H_\ell + \bu = H_\ell + \bv$ for all primes $\ell \neq p$.  At $p$, the isometry that sends $\be$ to $q^{-1}\be$ and $\mathbf f$ to $q\mathbf f$, whose determinant is 1,  would
send $H_p + \bu$ to $H_p + \bv$.  Suppose that there exists $\sigma \in O(W)$ which sends $H + \bu$ to $H + \bv$.    Then $\sigma$ necessarily sends $H$ to $H$ itself; hence the matrix for $\sigma$ relative to the basis $\{\be, \mathbf f\}$ is one of the
following:
$$\begin{bmatrix} 1 & 0\\ 0 & 1 \end{bmatrix},\quad
\begin{bmatrix} -1 & 0\\ 0 & -1 \end{bmatrix},\quad
\begin{bmatrix} 0 & 1\\ 1 & 0 \end{bmatrix},\quad
\begin{bmatrix} 0 & -1\\ -1 & 0 \end{bmatrix}.$$
But a simple calculation shows that none of the above sends $H + \bu$ to $H + \bv$.  Hence $H + \bu$ is not in the same class of $H + \bv$.  As a result, both $h^+(H + \bv)$ and $h(H + \bv)$ are greater than 1, while $h(H)$ and $h^+(H)$ are 1.
\end{exam}

Of course, there are $\ring$-cosets, which are not $\ring$-lattices themselves, whose class numbers are 1.  Here is an example:

\begin{exam}
Let $M$ be the $\mathbb Z$-lattice whose Gram matrix is $\langle 4, 4, 4 \rangle$ relative to a basis $\{\be, \mathbf f, \mathbf g\}$, and let $\bv$ be $\frac{\be + \mathbf f + \mathbf g}{2}$.  The class number of $M$ is 1.  The lattice $M + \mathbb Z \bv$ is isometric to
$$\begin{pmatrix}
3 & 1 & -1\\
1 & 3 & 1\\
-1 & 1 & 3
\end{pmatrix}$$
whose class number is also 1.  Since $h(M) = 1$, any $\z$-coset in $\text{gen}(M + \bv)$ has an isometric copy of the form $M + \bx$ for some $\bx \in \mathbb Q M$.   If $M + \bx \in \text{gen}(M + \bv)$, then the lattice $M + \mathbb Z \bx$ is in $\text{gen}(M + \mathbb Z\bv)$ which has only one class. Therefore, there exists an isometry $\sigma \in O(\mathbb Q M)$ such that $\sigma(\bx) \in M + \mathbb Z\bv$. Thus $\sigma(\bx) = \bz + a\bv$, where $\bz \in M$ and $a \in \mathbb Z$.  But $Q(\bx)$ must be odd; therefore $a$ must be odd and hence $\sigma(\bx) \equiv \bv$ mod $M$.  This shows that $\sigma(M + \bx) = M + \bv$ and so $h(M + \bv) = 1$.
\end{exam}

\begin{prop}
Let $\bx$ be a vector in $V$.  Suppose that for each $\p \in \Omega$, there exists $\sigma_\p \in O(V_\p)$ such that $\bx \in \sigma_\p(M_\p + \bv)$.  Then there exists $K + \bz \in \textnormal{gen}(M + \bv)$ such that $\bx \in K + \bz$.
\end{prop}
\begin{proof}
This follows from Lemma \ref{crt} since $\bx \in M_\p = M_\p + \bv$ for almost all $\p$.
\end{proof}

Let $a \in F$.  We say that $M + \bv$ represents $a$ if there exists a nonzero vector $\bz \in M + \bv$ such that $Q(\bz) = a$, and that $\gen(M + \bv)$ represents $a$ if $V_\ell$ represents $a$ for all places $\ell \not \in \Omega$ and $M_\p + \bv$ represents $a$ for all places $\p \in \Omega$.  The following corollary shows that the class number of a coset can be viewed as a measure of the obstruction of the local-to-global implication in (\ref{hasse}).

\begin{cor}
Let $a \in F^\times$. Suppose that $\textnormal{gen}(M_\p + \bv)$ represents $a$. Then there
exists $K + \bz \in \textnormal{gen}(M + \bv)$ which represents $a$.
\end{cor}
\begin{proof}
The hypothesis says that for each $\p \in \Omega$ there is a vector $\bz_\p \in M_\p + \bv$ such that $Q(\bz_\p) = a$.   By virtue of the Hasse Principle, there exists a vector $\bz \in V$ such that  $Q(\bz) = a$.  At each $\p\in \Omega$, it follows from Witt's extension theorem that there is an isometry $\sigma_\p \in O(V_\p)$ such that $\sigma_\p(\bz_\p) = \bz$.  Then for each $\p \in \Omega$,
$$\bz = \sigma_\p(\bz_\p) \in \sigma_\p(M_\p + \bv).$$
By the previous proposition, $\bz$ is contained in some coset $K + \bz \in \text{gen}(M + \bv)$.  Equivalently, $a$ is represented by $K + \bz$.
\end{proof}

When $F$ is a number field, the obstruction of the local-to-global principle for representations of cosets may be overcome by applying the results on representations of quadratic lattices with approximation properties.

\begin{thm} \label{rep}
Let $M+\bv$ be an $\ring$-coset on a positive definite quadratic space over a totally real number field $F$. Suppose that $a \in F^\times$ is represented by $\textnormal{gen}(M+\bv)$.
\begin{enumerate}
\item[(1)]  If $\dim(M) \geq 5$, then there exists a constant $C = C(M , \bv)$ such that $a$ is represented by $M + \bv$ provided $\mathbb N_{F/\q}(a) > C$.

\item[(2)] Suppose that $\dim(M) = 4$ and $a$ is primitively represented by $M_\p + \ring_\p\bv$ whenever $M_\p$ is anisotropic.  Then there exists a constant $C^* = C^*(M, \bv)$ such that $a$ is represented by $M + \bv$ provided $\mathbb N_{F/\q}(a) > C^*$.
\end{enumerate}
\end{thm}
\begin{proof}
(1) Let $S$ be the subset of $\Omega$ containing all $\p$ for which $M_\p + \bv \neq M_\p$ or $M_\p$ is not unimodular.  This $S$ is a finite set.  For each $\p \in S$, let $\bx_\p \in M_\p$ such that $Q(\bx_\p + \bv) = a$.  Choose an integer $s$ large enough so that $\p^s\bv \in M_\p$ for all $\p \in S$.  Let $C$ be the constant obtained from applying the number field version of the main theorem in \cite{jk}  (see \cite[Remark (ii)]{jk}) to $M + \ring\bv$, $S$, and $s$.  If $\mathbb N_{F/\q}(a) > C$, then there exists $\bw \in M + \ring \bv$ such that $Q(\bw) = a$ and $\bw \equiv \bx_\p + \bv \equiv \bv$ mod $\p^s(M_\p + \ring_\p \bv)$ for every $\p \in S$. Since $\p^s(M_\p + \ring_\p \bv) \subseteq M_\p$, it follows that $\bw$ is in $M + \bv$, which means that $M + \bv$ represents $a$.

Part (2) can be proved in the same manner, except that we need to replace the main theorem in \cite{jk} by \cite[Appendix]{ch}.
\end{proof}

When $V$ is indefinite, we need to take into account of the orthogonal complement of a vector representing $a$.  Since $a$ is represented by $\gen(M + \bv)$, $a$ must be represented by $\gen(M + \ring \bv)$, and it follows from the Hasse Principle that there exists $\bz \in V$ such that $Q(\bz) = a$.  Let $W$ be the orthogonal complement of $\bz$ in $V$.  The following theorem is an immediate consequence of \cite[Corollary 2.6]{bc}.

\begin{thm} \label{rep2}
Let $M + \bv$ be an $\ring$-coset on an indefinite quadratic space over a number field $F$.  Suppose that $a \in F^\times$ is represented by $\textnormal{gen}(M + \bv)$.
\begin{enumerate}
\item[(1)]  If $\dim(M) \geq 4$ or $W$ is isotropic, then $a$ is represented by $M + \bv$.

\item[(2)] Suppose that $\dim(M) = 3$, $W$ is anisotropic, and $M + \ring\bv$ represents $a$.  Then $M + \bv$ represents $a$ if either $a$ is a spinor exception of $\textnormal{gen}(M + \ring \bv)$ or there exists $\p \not \in \Omega$ for which $W_\p$ is anisotropic and additionally $V_\p$ is isotropic if $\p$ is a real place.
\end{enumerate}
\end{thm}
\begin{proof}
Let $T$ be the set of places $\p$ for which $\bv \not \in M_\p$.  By \cite[Corollary 2.6]{bc}, the hypothesis in either (1) or (2) implies that $M + \ring \bv$ represents $a$ with approximation at $T$.  Therefore, there exists $\bw \in M + \ring \bv$ such that $Q(\bw) = a$ and $\bw \equiv \bv$ mod $M_\p$ for all $\p \in T$.  Consequently, $M + \bv$ represents $a$.
\end{proof}

We conclude this paper by offering a few comments on the additional hypothesis placed in Theorem \ref{rep2}(2).   First, there is an effective procedure \cite{sp1} to decide whether $a$ is a spinor exception of $\gen(M + \ring \bv)$.  It depends on the knowledge of the local relative spinor norm groups $\theta(M_\p + \ring_\p\bv, a)$.  These groups have been computed in \cite{sp1} when $\p$ is nondyadic or 2-adic, and in \cite{x} when $\p$ is general dyadic.  When $a$ is a spinor exception of $\gen(M + \ring \bv)$,  it is also possible to determine if $M + \ring\bv$ itself represents $a$; see \cite[Theorem 3.6]{cx} for example.

\bibliographystyle{amsplain}

\end{document}